\documentclass[11pt,a4paper]{article}

\usepackage[english]{babel}

\usepackage{amsmath,amssymb}
\usepackage{color}
\usepackage{graphicx}
\usepackage{verbatim}

\usepackage{fullpage}
\newtheorem{theo}{Theorem}[section]
\newtheorem{propo}[theo]{Proposition}

\newtheorem{coro}[theo]{Corollary}

\newfont{\nset}{msbm10}

\def\A{\mbox{\boldmath $A$}}
\def\B{\mbox{\boldmath $B$}} 
\def\C{\mbox{\boldmath $C$}}
\def\D{\mbox{\boldmath $D$}}

\def\I{\mbox{\boldmath $I$}}

\def\M{\mbox{\boldmath $M$}}
\def\O{\mbox{\boldmath $O$}}

\def\dist{\mathop{\rm dist}\nolimits}
\def\ecc{\mathop{\rm ecc}\nolimits}

\def\>{\mathop{\rightarrow}\nolimits}

\def\u{{\mbox {\boldmath $u$}}}
\def\v{\mbox{\boldmath $v$}}

\def\x{\mbox{\boldmath $x$}}

\def\vec0{\mbox{\bf 0}}
\def\ev{\mathop{\rm ev}\nolimits}

\def\sp{\mathop{\rm sp}\nolimits}

\begin{document}

\title{On Middle Cube Graphs}

\author{C. Dalf\'o$^\dag$, M.A. Fiol$^\dag$, M. Mitjana$^\ddag$ \\
$^\dag${\small Departament de Matem\`atica Aplicada IV}\\
$^\ddag${\small Departament de Matem\`atica Aplicada I}\\
{\small Universitat Polit\`ecnica de Catalunya}\\
{\small {\tt \{cdalfo,fiol\}@ma4.upc.edu}}\\
{\small {\tt margarida.mitjana@upc.edu}}\\
}
\date {\today}

\maketitle

\begin{abstract}
We study a family of graphs related to the $n$-cube. The middle cube
graph of parameter $k$ is the subgraph of $Q_{2k-1}$ induced by the
set of vertices whose binary representation has either $k-1$ or $k$
number of ones. The middle cube graphs can be obtained from the
well-known odd graphs by doubling their vertex set. Here we study
some of the properties of the middle cube graphs in the light of the
theory of distance-regular graphs. In particular, we completely
determine their spectra (eigenvalues and their multiplicities, and
associated eigenvectors).

\end{abstract}

%%%%%%%%%%%%%%%%%%%%%%%%%%%%%%%%%%%%%%%%%%%%%%%%%

%%%%%%%%%%%%%%%%%%%%%%%%%%%%%%%%%%%%%%%%%%%%%%%%%
%  Introduction
%%%%%%%%%%%%%%%%%%%%%%%%%%%%%%%%%%%%%%%%%%%%%%%%%
\section{Introduction}

The {\em $n$-cube} $Q_n$, or $n$-dimensional hypercube, has been
extensively studied. Nevertheless, many open questions remain.
Harary {\em et al.} wrote a comprehensive survey on hypercube
graphs~\cite{HaHaWu88}. Recall that the {\em $n$-cube} $Q_n$ has
vertex set $V=\{0,1\}^n$ and $n$-tuples representing vertices are
adjacent if and only if they differ in exactly one coordinate. Then,
$Q_n$ is an $n$-regular bipartite graph with $2^n$ vertices and it
is natural to consider its vertex set as partitioned into $n+1$
layers, the {\em layer} $L_k$ consisting of the $n\choose k$
vertices containing exactly $k$ $1$s, $0\leq k\leq n$.
Seeing the vertices of $Q_n$ as the characteristic vector of subsets of $[n]=\{1,2,\dots,n\}$,
the vertices of layer $L_k$ correspond to the subsets of cardinality $k$,
while the adjacencies correspond to the inclusion relation.

If $n$ is odd, $n=2k-1$, the middle two layers $L_{k}$ and $L_{k-1}$
of $Q_n$ have the same number ${n \choose k}={n \choose k-1}$ of
vertices. Then the middle cube graph, denoted by $MQ_k$, is the
graph induced by these two layers. It has been conjectured by
Dejter, Erd\H{o}s, Havel~\cite{Ha83} among others, that $MQ_k$ is
Hamiltonian. It is known that the conjecture holds for $n\le 16$
(see Savage and Shields~\cite{SaSh99}), and it was {\em almost}
solved by Robert Johnson~\cite{Ro04}.

In this paper we study
some of the properties of the middle cube graphs in the light of the
theory of distance-regular graphs. In particular, we completely
determine their spectra (eigenvalues and their multiplicities, and
associated eigenvectors). In this context, Qiu and Das
provided experimental results for eigenvalues of several interconnection
networks for which no complete characterization were known (see \cite[\S3.2]{QiDa04}).

Before proceeding with our study, we fix some basic definitions and
notation used throughout the paper. We denote by $G=(V,E)$ a
(simple, connected and finite) {\em graph} with vertex set $V$ an
edge set $E$. The {\em order} of the graph $G$ is $n=|V|$ and its
{\em size} is $m=|E|$. We label the vertices with the integers
$1,2,\ldots, n$. If $i$ is adjacent to $j$, that is, $ij\in E$, we
write $i\sim j$ or $i\,\stackrel{\scriptscriptstyle{(E)}}{\sim}\,
j$. The {\it distance} between two vertices is denoted by $\dist
(i,j)$. We also use the concepts of {\it even distance} and {\it odd
distance }between vertices (see Bond and Delorme~\cite{BoDe88}),
denoted by $\dist^+$ and $\dist^-$, respectively. They are defined
as the length of a shortest even (respectively, odd) walk between
the corresponding vertices. The set of vertices which are {\it
$\ell$-apart} from vertex $i$, with respect to the usual distance,
is $\Gamma_\ell(i)=\{j:\dist (i,j)=\ell\}$, so that the {\em degree}
of vertex $i$ is simply $\delta_i:=|\Gamma_1(i)|\equiv |\Gamma(i)|$.
The {\it eccentricity} of a vertex is $\ecc(i):=\max_{1\le j\le
n}\dist (i,j)$ and the {\it diameter} of the graph is $D\equiv
D(G):=\max_{1\le i\le n}\ecc(i)$.
Given $0\leq \ell \leq D$, the {\em distance}-$\ell$ graph $G_\ell$
has the same vertex set as $G$ and two vertices are adjacent in
$G_\ell$ if and only if they are at distance $\ell$ in $G$. An
{\em antipodal graph} $G$ is a connected graph of diameter $D$ for
which $G_{D}$ is a disjoint union of cliques. In this case, the
{\em folded graph of} $G$ is the graph $\overline{G}$ whose
vertices are the maximal cliques of $G_D$ and two vertices are
adjacent if their union contains and edge of $G$. If, moreover, all
maximal cliques of $G_D$ have the same size $r$ then $G$ is also
called an {\em antipodal $r$-cover} of $\overline{G}$ (double
cover if $r=2$, triple cover if $r=3$, etc.).

Recall that a graph $G$ with diameter $D$ is
{\em distance-regular} when, for all integers $h,i,j$ ($0\le
h,i,j \le D$) and vertices $u,v\in V$ with $\dist(u,v)=h$, the
numbers
$$
p_{ij}^h=|\{w\in V: \dist(u,w)=i, \dist(w,v)=j \}|
$$
do not depend on $u$ and $v$. In this case, such numbers are
called the {\em intersection parameters} and, for notational
convenience, we write $c_i=p_{1i-1}^i$, $b_i=p_{1i+1}^i$, and
$a_i=p_{1i}^i$ (see Brower {\em et al.}~\cite{BrCoNe89} and
Fiol~\cite{Fi02}).

%%%%%%%%%%%%%%%%%%%%%%%%%%%%%%%%%%%%%%%%%%%%%%%%%
%  preliminaries
%%%%%%%%%%%%%%%%%%%%%%%%%%%%%%%%%%%%%%%%%%%%%%%%%
\section{Preliminaries}

%
%  odd graphs
%
\subsection{The odd graphs}

The odd graph, independently introduced by Balaban {\em et
al.}~\cite{BaFaBa66} and Biggs~\cite{Bi72}, is a family of graphs
that has been studied by many authors (see~\cite{Bi79,Bi93,Go80}).
More recently, Fiol {\em et al.}~\cite{FiGaYe00} introduced the
twisted odd graphs, which share some interesting properties with the
odd graphs although they have, in general, a more involved
structure.

For $k\ge 2$, {\em the odd graph} $O_k$ has vertices representing the
($k-1$)-subsets of $[2k-1]=\{1,2,\dots,2k-1\}$, and two vertices are
adjacent if and only if they are disjoint. For example, $O_2$ is the
complete graph $K_3$, and $O_3$ is the Petersen graph. In general,
$O_k$ is a $k$-regular graph on $n={{2k-1}\choose {k-1}}$ vertices,
diameter $D = k - 1$ and girth $g = 3$ if $k = 2$, $g = 5$ if $k =
3$, and $g = 6$ if $k > 4$ (see Biggs~\cite{Bi93}).

The odd graph $O_k$ is a distance-regular graph with intersection
parameters
$$
b_j=k-\left[\frac{j+1}{2}\right], \quad
c_j=\left[\frac{j+1}{2}\right] \quad (0\le j\le k-1).
$$

With respect to the spectrum, the distinct eigenvalues of $O_k$ are
$\lambda_i=(-1)^i(k-i)$, $0\le i \le k-1$, with multiplicities
$$
m(\lambda_i)={{2k-1}\choose i} - {{2k-1}\choose {i-1}}=
\frac{k-i}{k} {{2k} \choose i}.
$$

%
%  bipartite double graph
%
\subsection{The bipartite double graph}

Let $G=(V,E)$ be a graph of order $n$, with vertex set $V=\{1,2,\ldots, n\}$.
Its {\em bipartite double} graph $\widetilde{G}=(\widetilde{V},\widetilde{E})$
is the graph with the duplicated vertex set $\widetilde{V}=\{1,2,\ldots,n,1',2',\ldots,n'\}$,
and adjacencies induced from the adjacencies in $G$ as follows:
\begin{eqnarray}
\label{def.doble.bipartit} i \,
\stackrel{\scriptscriptstyle{(E)}}{\sim} \, j \Rightarrow \left\{
\begin{array}{l}
i\, \stackrel{\scriptscriptstyle{(\widetilde{E})}}{\sim}\, j', \textrm{ and}\\
j\, \stackrel{\scriptscriptstyle{(\widetilde{E})}}{\sim}\, i'.
\end{array}
\right.
\end{eqnarray}
Thus, the edge set of $\widetilde{G}$ is $\widetilde{E}=\{ij'|ij\in
E \}$.

From the definition, it follows that $\widetilde{G}$ is a bipartite
graph with stable subsets $V_1=\{1,2,\ldots,n\}$ and $V_2=\{1',2',\ldots,n'\}$.
For example, if $G$ is a bipartite graph, then its bipartite double graph
$\widetilde{G}$ consists of two non-connected copies of $G$ (see Fig.~\ref{fig.P4}).

%%%%%%%%%%%%%%%%%%%%%%%%%%%%%%%%%%%%%%%%%%%%%%%%
\begin{figure}[t]
\centering\includegraphics[scale=0.5]{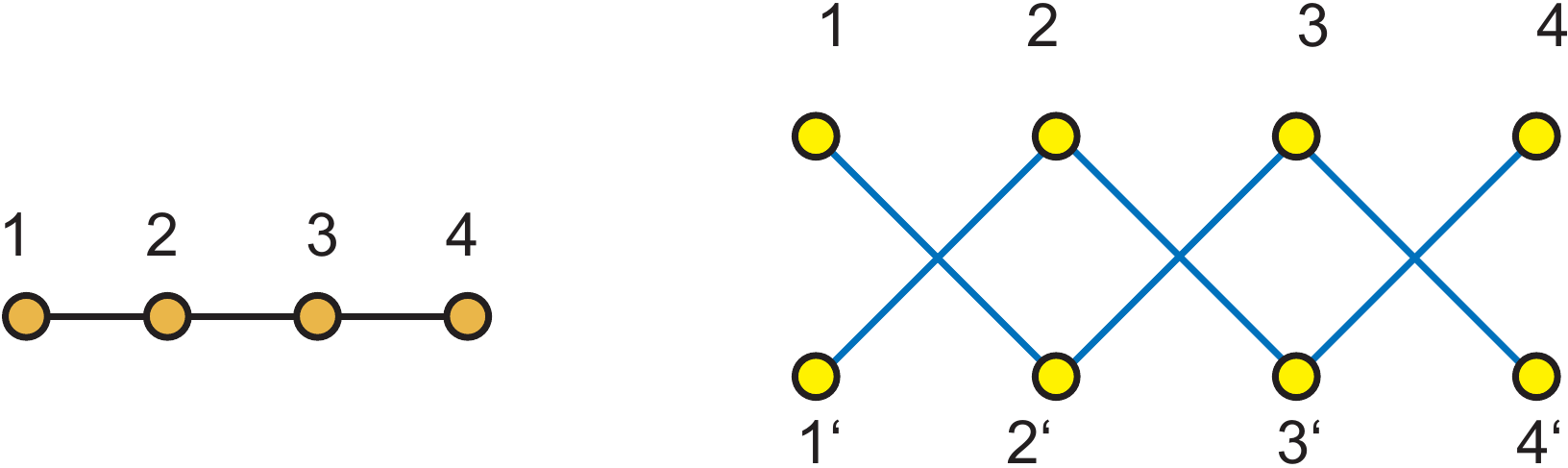} \hspace{1cm}
\caption{The path $P_4$ and its bipartite double graph.}
\label{fig.P4}
\end{figure}
%%%%%%%%%%%%%%%%%%%%%%%%%%%%%%%%%%%%%%%%%%%%%%%%

The bipartite double graph $\widetilde{G}$ has an involutive
automorphism without fixed edges, which interchanges vertices $i$
and $i'$. On the other hand, the map  from $\widetilde{G}$ onto $G$
defined by $i'\mapsto i,\ i\mapsto i$
is a $2$-fold covering.

If $G$ is a $\delta$-regular graph, then $\widetilde{G}$ also is.
Moreover, if the degree sequence of the original graph $G$ is
$\delta=(\delta_1,\delta_2,\ldots,\delta_n)$, the degree sequence
for its bipartite double graph is
$\widetilde{\delta}=(\delta_1,\delta_2,\ldots,
\delta_n,\delta_1,\delta_2,\ldots,\delta_n)$.

The distance between vertices in the bipartite double graph
$\widetilde{G}$ can be given in terms of the even and odd distances
in $G$. Namely,
\begin{eqnarray*}
\dist_{\widetilde{G}}(i,j) &=& \dist_G^{+}(i,j)\\
\dist_{\widetilde{G}}(i,j') &=& \dist_G^{-}(i,j).
\end{eqnarray*}
Note that always $\dist_G^{-}(i,j)>0$ even if $i=j$. Actually,
$\widetilde{G}$ is
connected if and only if $G$ is connected and non-bipartite.

More precisely, it was proved by Bond and Delorme~\cite{BoDe88} that
if $G$ is a non-bipartite graph with diameter $D$, then its
bipartite double graph $\widetilde{G}$ has diameter $\widetilde{D}\le2D+1$,
and $\widetilde{D}=2D+1$ if and only if for some vertex $i\in V$ the
subgraph induced by the vertices at distance less than $D$ from $i$,
$G_{\le D-1}(i)$, is bipartite.

%%%%%% comprovar que el par\`{a}graf anterior sigui cert %%%%%%

In Figs.~\ref{fig.LessDiam}-\ref{fig.dcPetersen},
we can see the bipartite double graph of three different graphs. The
cycle $C_5$ and Petersen graph both have diameter $D=2$, and their
bipartite double graphs have diameter $\widetilde{D}=2D+1=5$, while in
the first example (Fig.~\ref{fig.LessDiam}) $\widetilde{G}$ has
diameter $\widetilde{D}=3<2D+1$.

%%%%%%%%%%%%%%%%%%%%%%%%%%%%%%%%%%%%%%%%%%%%%%%%
\begin{figure}[t]
\centering
\includegraphics[scale=0.5]{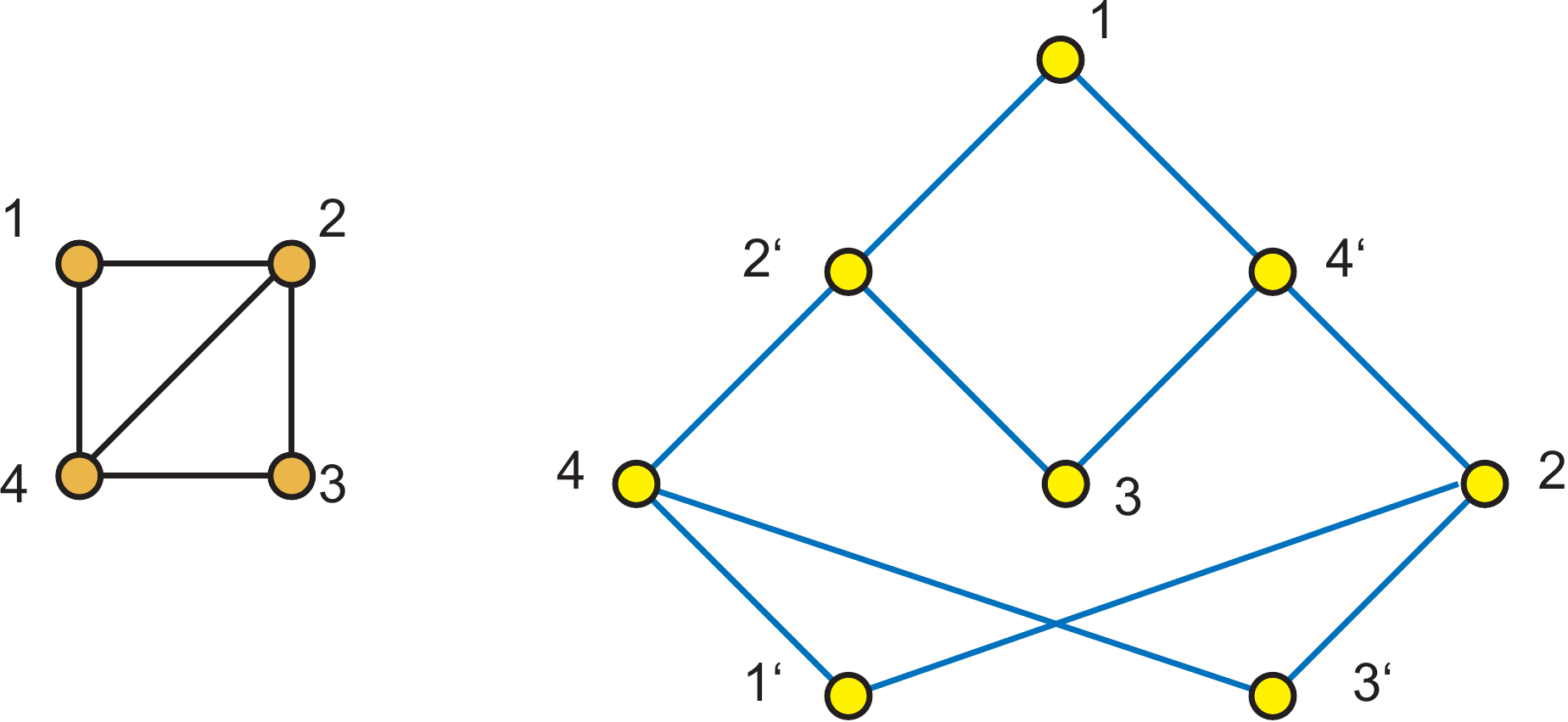}
\caption{Graph $G$ has diameter 2 and $\widetilde{G}$ has diameter 3.}
\label{fig.LessDiam}
\end{figure}
%%%%%%%%%%%%%%%%%%%%%%%%%%%%%%%%%%%%%%%%%%%%%%%%

%%%%%%%%%%%%%%%%%%%%%%%%%%%%%%%%%%%%%%%%%%%%%%%%
\begin{figure}[t]
\centering
\includegraphics[scale=0.5]{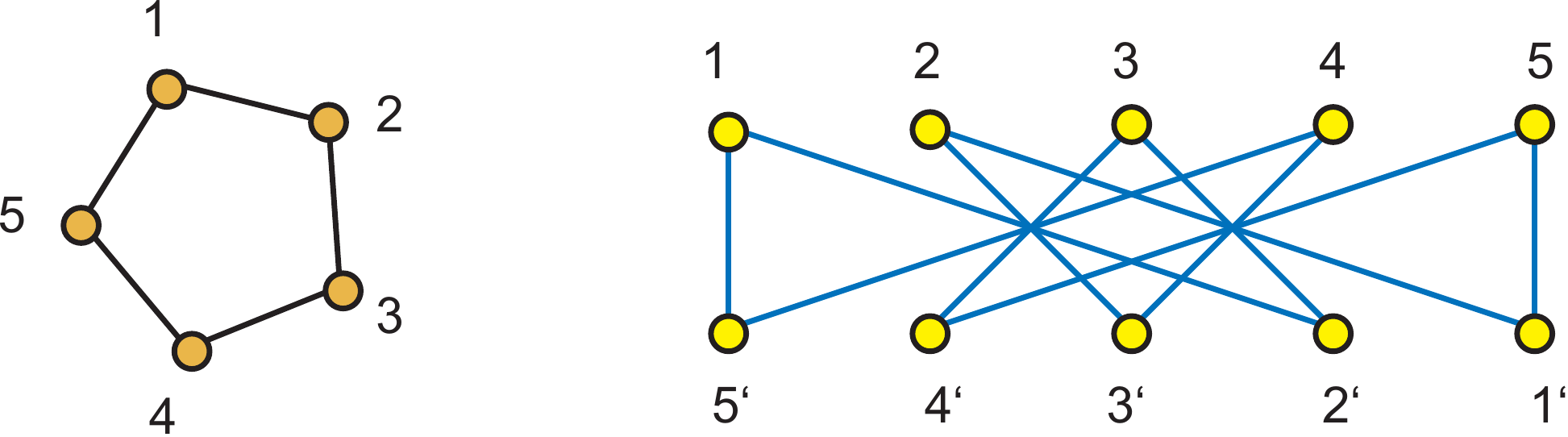}
\caption{$C_{5}$ and its bipartite double graph.}
\label{fig.C5}
\end{figure}
%%%%%%%%%%%%%%%%%%%%%%%%%%%%%%%%%%%%%%%%%%%%%%%%

%%%%%%%%%%%%%%%%%%%%%%%%%%%%%%%%%%%%%%%%%%%%%%%%
\begin{figure}[t]
\centering
\includegraphics[scale=0.5]{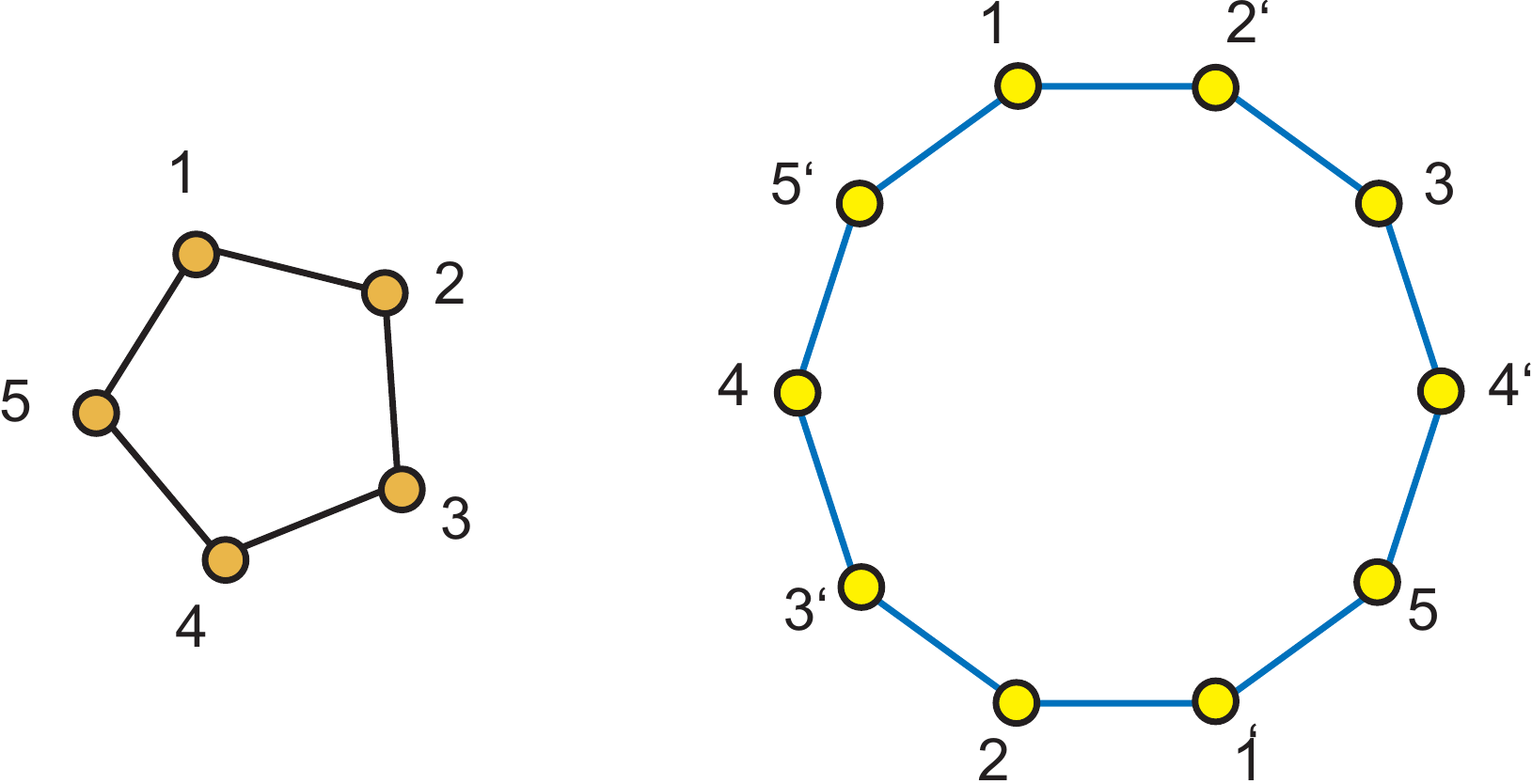}
\caption{$C_{5}$ and $C_{10}$ as another view of its bipartite double graph.}
\label{fig.C5.C10}
\end{figure}
%%%%%%%%%%%%%%%%%%%%%%%%%%%%%%%%%%%%%%%%%%%%%%%%

%%%%%%%%%%%%%%%%%%%%%%%%%%%%%%%%%%%%%%%%%%%%%%%%
\begin{figure}[t]
\centering
\centering\includegraphics[scale=0.5]{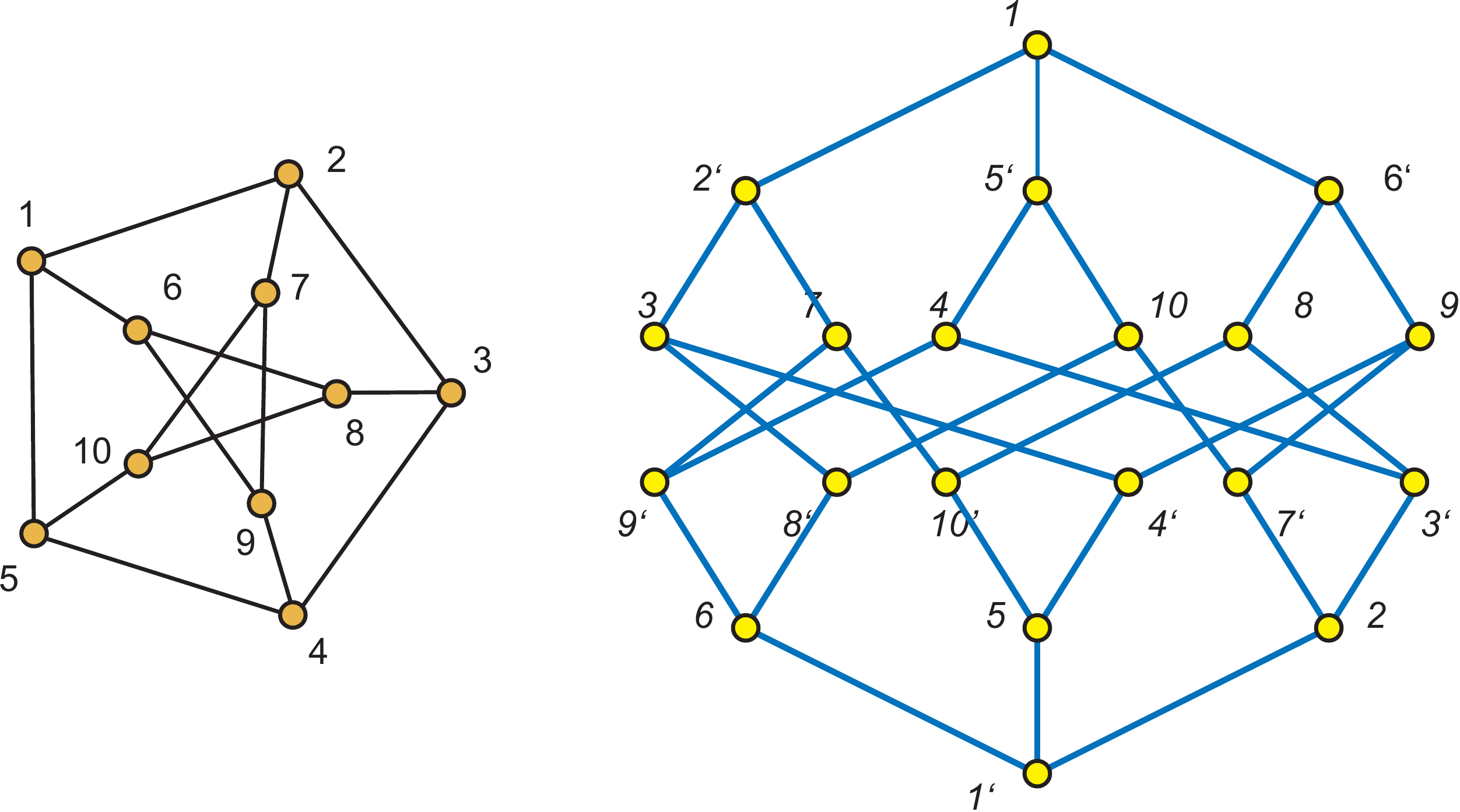}
\caption{Petersen's graph and its bipartite double graph.}
\label{fig.dcPetersen}
\end{figure}
%%%%%%%%%%%%%%%%%%%%%%%%%%%%%%%%%%%%%%%%%%%%%%%%

The {\em extended bipartite double} graph $\widehat{G}$ of a
graph $G$ is obtained from its bipartite double graph by adding
edges $(i,i')$ for each $i\in V$. Note that when $G$ is bipartite,
then $\widehat{G}$ is the direct product $G \Box K_2$.

%
% spectral properties of bipartite double graph
%
\subsection{Spectral properties of the bipartite double graph}

Let us now recall a useful result from spectral graph theory. For
any graph, it is known that the components of its eigenvalues can be
seen as charges on each vertex (see Fiol and Mitjana~\cite{FiMi07}
and Godsil~\cite{Go93}). Let $G=(V,E)$ be a graph with adjacency
matrix $\A$ and $\lambda$-eigenvector $\v$. Then, the charge of
vertex $i\in V$ is the entry $v_i$ of $\v$, and the equation
$\A\v=\lambda\v$ means that the sum of the charges of the neighbors
of vertex $i$ is $\lambda$ times the charge of vertex $i$:
$$
(\A\v)_i=\sum_{i \, \stackrel{\scriptscriptstyle{(E)}}{\sim} \, j}
v_j=\lambda v_i.
$$

In what follows we compute the eigenvalues of the bipartite double
graph $\widetilde{G}$ and the extended bipartite double graph
$\widehat{G}$ as functions of the eigenvalues of a non-bipartite
graph $G$. We also show how to obtain the eigenvalues together with
the corresponding eigenvectors of $\widetilde{G}$ and $\widehat{G}$.

First, we recall the following technical result, due to
Silvester~\cite{Si00}, on the determinant of some block matrices:

\begin{theo}
\label{lema.blocks} Let $F$ be a field and let $R$ be a commutative
subring of $F^{n\times n}$, the set of all $n\times n$ matrices over
$F$. Let $\M\in R^{m\times m}$, then
$$
det_F(\M)=det_F\big(det_R(\M)\big).
$$
\end{theo}

For example, if $\M=
\begin{pmatrix}
\A & \B \\
\C & \D
\end{pmatrix}
$, where $\A ,\B, \C, \D$ are $n\times n$ matrices over $F$ which
commute with each other, then Theorem~\ref{lema.blocks} reads
\begin{equation}
\label{determinant} det_F(\M)=det_F(\A\D-\B\C).
\end{equation}

Now we can use the above theorem to find the characteristic
polynomial of the bipartite double and the extended bipartite double
graphs.

\begin{theo}
Let $G$ be a graph on $n$ vertices, with the adjacency matrix $\A$ and
characteristic polynomial $\phi_G(x)$. Then, the characteristic
polynomials of $\widetilde{G}$ and $\widehat{G}$ are, respectively,
\begin{eqnarray}
\label{policarGtitlla}\phi_{\widetilde{G}}(x) &=& (-1)^n \phi_G(x)\phi_G(-x),\\
\label{policarGbarret}\phi_{\widehat{G}}(x) &=& (-1)^n
\phi_G(x-1)\phi_G(-x-1).
\end{eqnarray}
\end{theo}

\begin{proof}
From the definitions of $\widetilde{G}$ and $\widehat{G}$, their
adjacency matrices are, respectively,
$$
\widetilde{\A}=
\begin{pmatrix}
\O & \A \\
\A & \O
\end{pmatrix}
\ \mbox{ and } \ \widehat{\A}=
\begin{pmatrix}
\O & \A+\I \\
\A+\I & \O
\end{pmatrix}.
$$

\noindent Thus, by (\ref{determinant}), the characteristic
polynomial of
$\widetilde{G}$ is
\begin{eqnarray*}
\phi_{\widetilde{G}}(x) &=& \det (x\I_{2n}-\widetilde{\A}) = \det
\begin{pmatrix}
x\I_n & -\A \\
-\A & x\I_n
\end{pmatrix}
= \det(x^2\I_n-\A^2)\\
&=& \det(x\I_n-\A)\det(x\I_n+\A) = (-1)^n \phi_G(x)\phi_G(-x),
\end{eqnarray*}
whereas, the characteristic polynomial of $\widehat{G}$ is
\begin{eqnarray*}
\phi_{\widehat{G}}(x)& = & \det (x\I_{2n}-\widehat{\A}) = \det
\begin{pmatrix}
x\I_n & -\A-\I_n \\
-\A-\I_n & x\I_n
\end{pmatrix} \\
&=& \det\big(x^2\I_n-(\A+\I_n)^2\big)=
\det\big(x\I_n-(\A+\I_n)\big)\det\big(x\I_n+(\A+\I_n)\big)\\
&=& \det\big((x-1)\I_n-\A\big)(-1)^n\det\big(-(x+1)\I_n -\A\big)\\
&=& (-1)^n \phi_G(x-1)\phi_G(-x-1).
\end{eqnarray*}
\end{proof}

As a consequence, we have the following corollary:

\begin{coro}
\label{prop.dc.espectrum} Given a graph $G$ with spectrum
$$
\sp G=\{\lambda_0^{m_0},\lambda_1^{m_1}, \dots, \lambda_d^{m_d} \},
$$
where the superscripts denote multiplicities, then the spectra of
$\widetilde{G}$ and $\widehat{G}$ are, respectively,
\begin{eqnarray*}
\sp\widetilde{G} &=& \{\pm \lambda_0^{m_0},\pm \lambda_1^{m_1},
\ldots, \pm\lambda_d^{m_d} \},\\
\sp\widehat{G} &=&
\{\pm(1+\lambda_0)^{m_0},\pm(1+\lambda_1)^{m_1},\ldots,
\pm(1+\lambda_d)^{m_d}\}.
\end{eqnarray*}
\end{coro}

\begin{proof}
Just note that, by (\ref{policarGtitlla}) and
(\ref{policarGbarret}), for each root $\lambda$ of $\phi_G(x)$,
$\mu=\pm\lambda$ are roots of $\phi_{\widetilde{G}}(x)$, whereas
$\mu=\pm(1+\lambda)$ are roots of $\phi_{\widehat{G}}(x)$.
\end{proof}

Note that the spectra of $\widetilde{G}$ and $\widehat{G}$ are
symmetric, as expected, because both $\widetilde{G}$ and
$\widehat{G}$ are bipartite graphs.

In the next theorem we are concerned with the eigenvectors of
$\widetilde{G}$ and $\widehat{G}$, in terms of the eigenvectors of $G$.
The computations also give an alternative derivation of the above spectra.
% can be carried out by
%computing the eigenvectors of $\widetilde{G}$ and $\widehat{G}$.
\begin{theo}
Let $G$ be a graph and $\v$ a $\lambda$-eigenvector of $G$.
Let us consider the vector $\u^+$ with components
$u_i^+=u_{i'}^+=v_{i}$, and $\u^-$, with components
$u_i^-=v_i$ and $u_{i'}^-=-v_{i}$, $1 \leq i,i' \leq n$.
Then,
\begin{itemize}
\item $\u^+$ is a $\lambda$-eigenvector of $\widetilde{G}$ and
a $(1+\lambda)$-eigenvector of $\widehat{G}$;
\item $\u^-$ is a $(-\lambda)$-eigenvector of $\widetilde{G}$ and
a $(-1-\lambda)$-eigenvector of $\widehat{G}$.
\end{itemize}
\end{theo}

\begin{proof}
In order to show that
$\u^+$ is a $\lambda$-eigenvector of $\widetilde{G}$, we distinguish
two cases:
\begin{itemize}
\item
For a given vertex $i$, $1\le i\le n$, all its adjacent vertices are
of type $j'$, with $i \, \stackrel{\scriptscriptstyle{(E)}}{\sim} \,
j$. Then
$$
(\A\u^+)_i= \sum_{j' \,
\stackrel{\scriptscriptstyle{(\widetilde{E})}}{\sim} \,
i}u_{j'}^+=\sum_{j \, \stackrel{\scriptscriptstyle{(E)}}{\sim} \, i}
v_j= \lambda v_i= \lambda u_i^+.
$$
\item
For a given vertex $i'$, $1\le i\le n$, all its adjacent vertices are
of type $j$, with $i \, \stackrel{\scriptscriptstyle{(E)}}{\sim} \,
j$. Then
$$
(\A\u^+)_{i'}= \sum_{j \,
\stackrel{\scriptscriptstyle{(\widetilde{E})}}{\sim} \,
i'}u_{j}^+=\sum_{j \, \stackrel{\scriptscriptstyle{(E)}}{\sim} \, i}
v_j= \lambda v_{i}= \lambda u_i^+.
$$
\end{itemize}

By a similar reasoning with $u^-$, we obtain
$$
(\A\u^-)_i =
\sum_{j' \,\stackrel{\scriptscriptstyle{(\widetilde{E})}}{\sim} \, i}u_{j'}^- =
-\sum_{j \, \stackrel{\scriptscriptstyle{(E)}}{\sim} \, i} v_j =
%-\lambda v_i =
-\lambda u_i^-
\,\textrm{ and }\,
(\A\u^-)_{i'} =
\sum_{j \,\stackrel{\scriptscriptstyle{(\widetilde{E})}}{\sim} \, i'}u_{j}^- =
\sum_{j \, \stackrel{\scriptscriptstyle{(E)}}{\sim} \, i} v_j =
%\lambda v_{i}=
-\lambda u_{i'}^-.
$$
Therefore, $\u^-$ is a ($-\lambda$)-eigenvector of the
bipartite double graph $\widetilde{G}$.

In the same way, we can prove that $\u^+$ and $\u^-$ are
eigenvectors of $\widehat{G}$ with respective eigenvalues
$1+\lambda$ and $-1-\lambda$.
\end{proof}

Notice that, for every linearly independent eigenvectors $\v_1$ and
$\v_2$ of $G$, we get the linearly independent eigenvectors
$\u_1^{\pm}$ and $\u_2^{\pm}$ of $\widetilde{G}$. As a consequence,
the geometric multiplicity of eigenvalue $\lambda$ of $G$ coincides
with the geometric multiplicities of the eigenvalues $\lambda$ and
$-\lambda$ of $\widetilde{G}$, and $1+\lambda$ and $-1-\lambda$ of $\widehat{G}$.

%%%%%%%%%%%%%%%%%%%%%%%%%%%%%%%%%%%%%%%%%%%%%%%%%
% micus
%%%%%%%%%%%%%%%%%%%%%%%%%%%%%%%%%%%%%%%%%%%%%%%%%
\section{The middle cube graphs}

For $k \geq 1$ and $n = 2k - 1$,  the {\em middle cube graph} $MQ_k$
is the subgraph of the $n$-cube $Q_n$ induced by the vertices whose
binary representations have either $k-1$ or $k$ number of $1$s.
Then, $MQ_k$ has order $2{n\choose{k}}$ and is $k$-regular, since a
vertex with $k-1$ $1$s has $k$ zeroes, so it is adjacent to $k$
vertices with $k$ $1$s, and similarly a vertex with $k$ $1$s has $k$
adjacent vertices with $k-1$ $1$s (see Figs.~\ref{fig.3midcub}
and~\ref{fig.5midcub}).

The middle cube graph $MQ_k$ is a bipartite graph with stable
sets $V_0$ and $V_1$ constituted by the vertices whose corresponding
binary string has, respectively, even or odd {\em Hamming weight},
that is, number of $1$s. The diameter of the middle cube graph
$MQ_k$ is $D=2k-1$.

%
%  micus i bipartite double dels odd
%
\subsection{$MQ_k$ is the bipartite double graph of $O_k$}

Notice that, if $A$ and $B$ are both subsets of $[2k-1]$, $A\subset
B$ if and only if $A$ and $\overline{B}$ are disjoint. Moreover, if
$|B|=k$ then $|\overline{B}|=k-1$. This gives the following result.
\begin{propo}
The middle cube graph $MQ_k$ is isomorphic to $\widetilde{O}_k$,
the bipartite double graph of $O_k$.
\end{propo}
\begin{proof}
The mapping from $\widetilde{O}_k$ to $MQ_k$ defined by:
$$
\begin{array}{rcccl}
f: & V[\widetilde{O}_k] & \to & V[MQ_k] \\
& \u &\mapsto & \u \\
& \u' &\mapsto & \overline{\u}
\end{array}
$$
is clearly bijective. Moreover, according to the definition of
bipartite double graph in Eq.(\ref{def.doble.bipartit}), if $\u$ and $\v'$ are two vertices of
$\widetilde{O}_k$, then
$$
\u\sim\v' \Leftrightarrow \u\cap\v=\varnothing \Leftrightarrow
\u\subset\overline{\v},
$$
which is equivalent to say that
if $\u\sim\v'$, in $\widetilde{O}_k$, then $f(\u)=\u\sim \overline{\v}=f(\v')$, in $MQ_k$.
\end{proof}
For example, the middle cube graph $MQ_2$ contains vertices with one
or two $1$s in their binary representation. The adjacencies give simply a
6-cycle (see Fig.~\ref{fig.3midcub}), which is isomorphic to
$\widetilde{O}_2$. As another example, $MQ_3$ has 20 vertices
because there are ${5\choose 2} =10$ vertices with two $1$s, and
${5\choose3} =10$ vertices with three $1$s in their binary
representation (see Fig.~\ref{fig.5midcub}).
Compare the Figs.~\ref{fig.dcPetersen} and~\ref{fig.5midcub} in
order to realize the isomorphism between the definitions of $MQ_3$
and $\widetilde{O}_3$.

It is known that $\widetilde{O}_k$ is a bipartite 2-antipodal
distance-regular graph. See Biggs~\cite{Bi93} and
Brower {\em et al.}~\cite{BrCoNe89} for more details.

%%%%%%%%%%%%%%%%%%%%%%%%%%%%%%%%%%%%%%%%%%%%%%%%
\begin{figure}[t]
\centering
\includegraphics[scale=0.5]{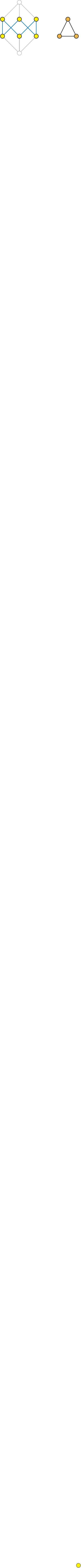}
\caption{The middle cube graph $MQ_2$ as a subgraph of $Q_3$ or as
the bipartite double graph of $O_2=K_3$.} \label{fig.3midcub}
\end{figure}
%%%%%%%%%%%%%%%%%%%%%%%%%%%%%%%%%%%%%%%%%%%%%%%%

%%%%%%%%%%%%%%%%%%%%%%%%%%%%%%%%%%%%%%%%%%%%%%%%
\begin{figure}[t]
\centering
\includegraphics[scale=0.5]{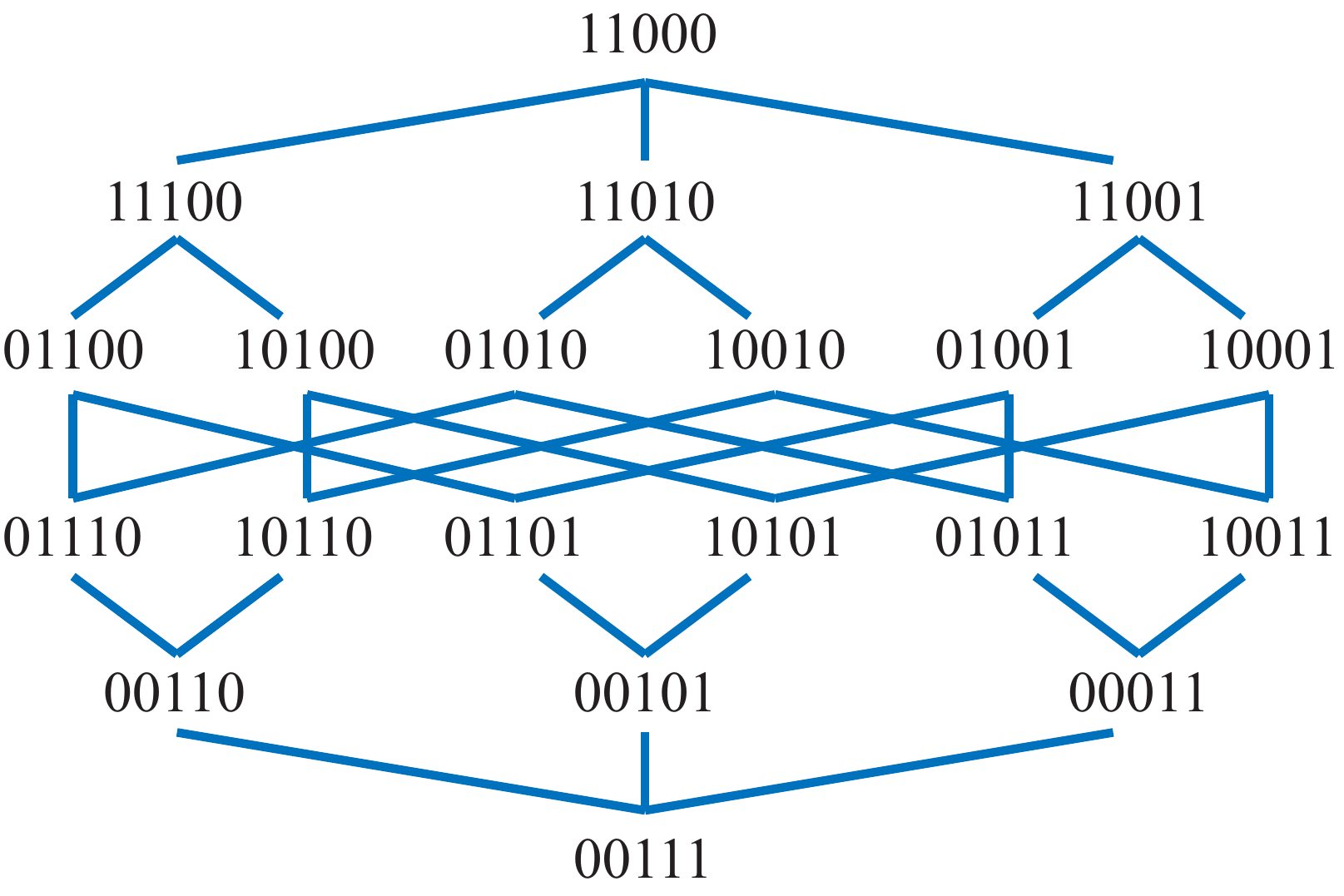}
\caption{The middle cube graph $MQ_3$.} \label{fig.5midcub}
\end{figure}
%%%%%%%%%%%%%%%%%%%%%%%%%%%%%%%%%%%%%%%%%%%%%%%%

%
%  spectral properties
%
\subsection{Spectral properties}

The spectrum of the hypercube $Q_{2k-1}$ contains all the
eigenvalues (including multiplicities) of the middle cube $MQ_k$:
$$
\sp MQ_k\subseteq \sp Q_{2k-1}.
$$

According to the result of Corollary~\ref{prop.dc.espectrum}, the
spectrum of the middle cube graph $MQ_k\simeq\widetilde{O}_k$ can be
obtained from the spectrum of the odd graph $O_{k}$. The distinct
eigenvalues of $MQ_k$ are $\theta_i^+=(-1)^i(k-i)$ and
$\theta_i^-=-\theta_i^+$, $0\le i \le k-1$, with multiplicities
\begin{equation}
\label{mult_vaps_micu}
m(\theta_i^+)=m(\theta_i^-)=\frac{k-i}{k}
{{2k} \choose i}.
\end{equation}

For example,
\begin{eqnarray*}
\sp MQ_3 &=& \{ \pm 2,\pm 1^2 \},\\
\sp MQ_5 &=& \{\pm 3,\pm 2^4, \pm 1^5 \},\\
\sp MQ_7 &=& \{\pm 4,\pm 3^6,\pm 2^{14}, \pm 1^{14}\},\\
\sp MQ_9 &=& \{\pm 5, \pm 4^8, \pm 3^{27},\pm 2^{48}, \pm 1^{42}\}.\\
\end{eqnarray*}
The middle cube graph is a distance-regular graph. For instance, the
distance polynomials of $MQ_k$ are
\begin{eqnarray*}
p_0(x) &=& 1,\\
p_1(x) &=& x,\\
p_2(x) &=& x^2-3,\\
p_3(x) &=& \frac{1}{2}(x^3-5x),\\
p_4(x) &=& \frac{1}{4}(x^4-9x^2+12),\\
p_5(x) &=& \frac{1}{12}(x^5-11x^3+22x).
\end{eqnarray*}

As the sum of the distance polynomials is the Hoffman polynomial~\cite{Ho63}, we
have
\begin{equation}
\sum_{i=0}^5 p_i(x)=\frac{1}{12}(x-1)(x-2)(x+3)(x+2)(x+1).
\label{Eq.Hoffmann}
\end{equation}

\noindent The eigenvalues of the $MQ_3$ are $\lambda_0=3$ and the
zeroes of polynomial (\ref{Eq.Hoffmann}):
$$
\ev MQ_3=\{ 3,2,1,-1,-2,-3\},
$$
and their multiplicities, $m(\lambda_{i})$, can be computed
using the highest degree polynomial $p_{2k-1}$, according to the
result by Fiol~\cite{Fi02}:
$$
m(\lambda_{i})=\frac{\phi_0 p_{2k-1}(\lambda_0)}{\phi_i
p_{2k-1}(\lambda_i)}, \qquad 0 \le i \le {2k-1},
$$
where $\phi_i= \prod_{j=0, \,j\ne i}^{2k-1}(\lambda_i-\lambda_j)$.
Of course, this expression yields the same result as
Eq.~(\ref{mult_vaps_micu}). Namely, $m(\lambda_{i}) =
m(\lambda_{2k-1-i}) = m(\theta_{i}^{\pm})$, $0\leq i\leq k-1$.

The values of the highest degree polynomial are
$p_5(3)=p_5(1)=p_5(-1)=1$ and $p_5(2)=p_5(-1)=p_5(-3)=-1$. Moreover,
$\phi_0=-\phi_5=240$, $\phi_1=-\phi_4=-60$, and $\phi_2=-\phi_3=48$.
Then,
$$
\begin{array}{c}
m(\lambda_{0}) = m(\lambda_{5}) = m(\theta_{0}^{\pm}) = 1, \\
m(\lambda_{1}) = m(\lambda_{4}) = m(\theta_{1}^{\pm}) =4, \\
m(\lambda_{2}) = m(\lambda_{3}) = m(\theta_{2}^{\pm})= 5.
\end{array}
$$

%
%  micus boundary
%
\subsection{Middle cube graphs as boundary graphs}

Let $G$ be a graph with diameter $D$ and distinct eigenvalues $\ev
G=\{\lambda_0,\lambda_1,\ldots,\lambda_d\}$, where
$\lambda_0>\lambda_1>\dots>\lambda_d$. A classical result states
that $D\leq d$ (see, for instance, Biggs~\cite{Bi93}). Other results
related to the diameter $D$ and some (or all) different eigenvalues
have been given by Alon and Milman~\cite{AlMi85}, Chung~\cite{Ch89},
van Dam and Haemers~\cite{vDHa95}, Delorme and Sol\'{e}~\cite{DeSo91},
and Mohar~\cite{Mo91}, among others. Fiol {\em et
al.}~\cite{FiGaYe96,FiGaYe01,FiGaYe98b} showed that many of these
results can be stated with the following common framework: If the
value of a certain polynomial $P$ at $\lambda_0$ is large enough,
then the diameter is at most the degree of $P$. More precisely, it
was shown that optimal results arise when $P$ is the so-called {\em
$k$-alternating polynomial}, which in the case of degree $d-1$ is
characterized by $P(\lambda_i)=(-1)^{i+1}, 1\le i \le d$, and
satisfies $P(\lambda_0)=\sum_{i=1}^{d}\frac{\pi_0}{\pi_i}$, where
$\pi_i=\prod_{j=0, j\ne i }^{d}|\lambda_i-\lambda_j|$. In
particular, when $G$ is a regular graph on $n$ vertices, the
following implication holds:
$$
P(\lambda_0)+1= \sum_{i=0}^n\frac{\pi_0}{\pi_i}>n \quad\Rightarrow\quad D \le
d-1.
$$
This result suggested the study of the so-called
{\em boundary graphs}~\cite{FiGaYe98a,FiGaYe98b}, characterized by
\begin{equation}
\label{boundary}
\sum_{i=1}^{d}\frac{\pi_0}{\pi_i} = n.
\end{equation}
Fiol {\em et al.}~\cite{FiGaYe98a} showed that extremal ($D = d$)
boundary graphs, where each vertex has maximum eccentricity, are
2-antipodal distance-regular graphs. As we show in the next result,
this is the case of the middle cube graphs $MQ_{k}$ where the
antipodal pairs of vertices are $(\x;\overline{\x})$, with
$\x=x_0x_1 \ldots x_{2k-1}$ and
$\overline{\x}=\overline{x}_0\,\overline{x}_1 \ldots
\overline{x}_{2k-1}$.

\begin{propo}
The middle cube graph $MQ_{k}$ is a boundary graph.
\end{propo}
\begin{proof}
Recall that the eigenvalues of $MQ_{k}$ are
$$
\ev MQ_{2k-1}=\{k,k-1,\ldots, 1,-1,\ldots,-k\},
$$
that is, $\lambda_i=k-i, \lambda_{k+i}=-(i+1)$, $0\le i<k$. Now,
according to Eq.~(\ref{boundary}), we have to prove that
$\sum_{i=0}^{2k-1} \frac{\pi_0}{\pi_i}= 2 {{2k-1}\choose{k}}$.
Computing $\pi_i$, for $0\le i\le 2k-1$, we get
$$
\pi_i=\frac{i!(2k-i)!}{k-i}=\pi_{2k-(i+1)}, \textrm{ for } 0\le i<k.
$$
This implies
$$
\frac{\pi_0}{\pi_i}=\frac{\pi_0}{\pi_{2k-(i+1)}}=\frac{(2k)!}{k} \frac{(k-i)}{i!\,(2k-i)!}=
\frac{k-i}{k}{{2k}\choose{i}}, \textrm{ for } 0\le i<k,
$$
giving exactly the multiplicities of the corresponding eigenvalues,
as found in Eq.~\ref{mult_vaps_micu}.
By summing up we get
\begin{equation}
\label{eq.calcul0} \sum_{i=0}^{2k-1} \frac{\pi_0}{\pi_i} =
2\sum_{i=0}^{k-1} \frac{\pi_0}{\pi_i} = 2\left(
\sum_{i=0}^{k-1}{{2k}\choose{i}}-
\sum_{i=1}^{k-1}\frac{i}{k}{{2k}\choose{i}}\right).
\end{equation}
But
\begin{eqnarray}
\label{eq.calcul} \sum_{i=0}^{k-1}{{2k}\choose{i}} =
\frac{1}{2}\left(2^{2k}-{{2k}\choose{k}}\right) =
2^{2k-1}-{{2k-1}\choose{k}},
\end{eqnarray}
and
\begin{equation*}
\sum_{i=1}^{k-1}\frac{i}{k}{{2k}\choose{i}} =
2\sum_{i=0}^{k-2}{{2k-1}\choose{i}} =
2^{2k-1}-2{{2k-1}\choose{k-1}},
\end{equation*}
where we have used Eq.~(\ref{eq.calcul}) changing $k$ by $k-1$.
Thus, replacing the above values in Eq.~(\ref{eq.calcul0}), we get
the result.
\end{proof}

\subsection*{Acknowledgements} Research supported by the
{\em Ministerio de Ciencia e Innovaci\'on}, Spain, and the
{\em European Regional Development Fund} under project MTM2011-28800-C02-01,
and the {\em Catalan Research Council} under project 2009SGR1387.

%%%%%%%%%%%%%%%%%%%%%%%%%%%%%%%%%%%%%%%%%%%%%%%%%
%  biblio
%%%%%%%%%%%%%%%%%%%%%%%%%%%%%%%%%%%%%%%%%%%%%%%%%

\end{document}